\documentclass[11pt, english]{amsart}
\usepackage[english]{babel}
\usepackage[utf8x]{inputenc}

\usepackage{amsmath, amssymb, amsfonts}
\usepackage{amsthm}
\usepackage{csquotes}
\usepackage{theoremref}
\usepackage{auto-pst-pdf}
\usepackage{pstricks-add}
\usepackage{pst-all}
\usepackage{xcolor}
\usepackage{enumerate}
\usepackage{enumitem}

\theoremstyle{plain}
\newtheorem{theorem}{Theorem}[section]
\newtheorem{proposition}[theorem]{Proposition}
\newtheorem{corollary}[theorem]{Corollary}
\newtheorem{lemma}[theorem]{Lemma}

\theoremstyle{definition}
\newtheorem{definition}[theorem]{Definition}

\theoremstyle{remark}
\newtheorem*{example}{Example}
\newtheorem*{remark}{Remark}

\DeclareMathOperator{\Aut}{Aut}

\newcommand{\bg}{\mathcal{B}(G)}
\newcommand{\preline}{\mathcal{L}_U}
\newcommand{\aut}{\mathrm{Aut}(G)}
\newcommand{\up}{U_+}
\newcommand{\upp}{U_{++}}
\newcommand{\um}{U_-}
\newcommand{\umm}{U_{--}}

\title{The scale function and tidy subgroups}
\author[A. Brehm]{Albrecht Brehm}
\author[M. Gheysens]{Maxime Gheysens}
\author[A. Le Boudec]{Adrien Le Boudec}
\author[R. Rollin]{Rafaela Rollin}
\date{November 27, 2015}

\begin{document}

\begin{abstract}
This is an introduction to the structure theory of totally disconnected locally compact groups initiated by Willis in 1994. The two main tools in this theory are the scale function and tidy subgroups, for which we present several properties and examples.

As an illustration of this theory, we give a proof of the fact that the set of periodic elements in a totally disconnected locally compact group is always closed, and that such a group cannot have ergodic automorphisms as soon as it is non-compact.
\end{abstract}

\maketitle

\setcounter{tocdepth}{1}
\tableofcontents

\noindent

\section{Introduction}

Let $G$ be a locally compact totally disconnected group. Recall that this means that $G$ is a group endowed with a group topology that is locally compact, and such that connected components are singletons. Examples of such groups include for instance the automorphism group $\mathrm{Aut}(T_d)$ of a regular tree $T_d$ of degree $d \geq 3$, or the general linear group $\mathrm{GL}_n(\mathbb{Q}_p)$ endowed with its $p$-adic topology. We will denote by $\aut$ the group of bicontinuous automorphisms of the group $G$.

\subsection*{Scale function} Recall that according to van Dantzig's theorem, compact open subgroups of $G$ exist and form a basis of neighbourhoods of the identity in $G$. We will denote by $\bg$ the collection of compact open subgroups of $G$. For $U \in \bg$ and $\alpha \in \aut$, observe that $\alpha(U) \cap U$ is open in $G$, and since $U$ is compact, this implies that $\alpha(U) \cap U$ has finite index in $\alpha(U)$. The \emph{scale function} is defined as  \[s\colon \aut \to \mathbb N, \quad \alpha \mapsto \min_{U\in \bg}|\alpha(U):\alpha(U)\cap U|.\] By identifying an element $g \in G$ with the inner automorphism $x \mapsto gxg^{-1}$, we obtain a function on $G$ \[s\colon G \to \mathbb N, \quad  g \mapsto \min_{U \in \bg}|gUg^{-1}:gUg^{-1}\cap U|.\]

We will call a compact open subgroup $U \in \bg$ \emph{minimising} for $\alpha \in \aut$ if \[s(\alpha)=|\alpha(U):\alpha(U)\cap U|,\] i.e. if the minimum is attained at $U$. 

The scale function was introduced by Willis in \cite{Wil94}, and was subsequently used to answer questions in different fields of mathematics. We refer the reader to the introduction of \cite{Wil15} and to references therein for more details. Here we aim to give an account of this theory, with an emphasis on examples (see Section \ref{sec-ex}).

\subsection*{Tidy subgroups} For $\alpha \in \aut$ and $U \in \bg$, we let \[ \up = \bigcap_{n \geq 0} \alpha^n(U) \, \, \, \text{and} \, \, \, \um = \bigcap_{n \geq 0} \alpha^{-n}(U). \]

Note that by definition \mbox{$U \cap \alpha(\up) = \up$}. We therefore have an injective map \[\chi\colon \alpha(U_+)/U_+ \to \alpha(U)/\alpha(U)\cap U,\] and in particular $|\alpha(\up):\up| \leq |\alpha(U):\alpha(U)\cap U|$. This motivates the following definition.

\begin{definition}
A compact open subgroup $U$ is called \emph{tidy above} for $\alpha$ if $|\alpha(\up):\up| = |\alpha(U):\alpha(U)\cap U|$. 
\end{definition}

We also let \[ \upp = \bigcup_{n \geq 0} \alpha^n(\up) \, \, \, \text{and} \, \, \, \umm = \bigcup_{n \geq 0} \alpha^{-n}(\um). \] Note that these are increasing unions, and it follows that $\upp$ and $\umm$ are subgroups of $G$.

\begin{definition}
A compact open subgroup $U$ is called \emph{tidy below} for $\alpha$ if $U_{++}$ and $U_{--}$ are closed in $G$.
\end{definition}

\begin{definition}
A compact open subgroup is called \emph{tidy} for $\alpha$ if it is both tidy above and tidy below for $\alpha$.
\end{definition}

One of the reasons why this notion of tidy subgroups is relevant in this context is that it precisely describes the properties that minimising subgroups must have in common (see Theorem \ref{thm-mini-tidy}).

\begin{remark}
As mentioned earlier, one can consider an element $g \in G$ as the induced inner automorphism of $G$. Conversely, an automorphism $\alpha$ can be interpreted as conjugation in the group $\mathbb{Z} \ltimes_{\alpha} G$, so that we will freely switch between the two terminologies.
\end{remark}

\section{The structure of minimising subgroups}

\subsection{First properties of the scale function}

In this paragraph we establish some elementary properties of the scale function, using hardly more than its definition. The proofs are mainly taken from \cite{Moller}.

Recall that if $\mu$ is a left-invariant Haar measure on $G$, then for every $g \in G$ there exists a unique positive real number $\Delta(g)$ such that $\mu(Ag^{-1}) = \Delta(g) \mu (A)$ for every Borel subset $A$ of $G$. The function $\Delta$ is called the modular function on $G$, and is a continuous group homomorphism from $G$ to the multiplicative group of positive real numbers. There is an easy way to express $\Delta$ using a compact open subgroup of $G$.

\begin{lemma} \label{lemma-modular}
For all $g \in G$ and for all $U \in \mathcal{B}(G)$, the modular function is given by $$ \Delta(g) = \frac{|U : U \cap g^{-1}Ug|}{|U: U \cap gUg^{-1}|}.$$
\end{lemma}

\begin{proof} The proof is a simple computation: 
\begin{align*} |U:U\cap g^{-1}U g| &= |gUg^{-1} : gUg^{-1} \cap U| \\ 
&= \frac{\mu(gUg^{-1})}{\mu(gUg^{-1} \cap U)}\\ &= \frac{\mu(U) \Delta(g)}{\mu(gUg^{-1} \cap U)}  \\ &= |U : U \cap gUg^{-1}| \cdot \Delta(g). \end{align*}
\end{proof}

Next we will use this representation of the modular function to connect it to the scale.

\begin{proposition} \label{prop-modular}
The relation $\displaystyle{\frac{s(g)}{s(g^{-1})} = \Delta(g)}$ holds for every $g\in G$.
\end{proposition}

\begin{proof}
Let $g \in G$ and $U_1, U_2 \in \mathcal{B}(G)$ be such that $s(g) = |U_1 : U_1 \cap g^{-1}U_1g|$ and $s(g^{-1})=|U_2 : U_2 \cap gU_2g^{-1}|$. Since $U_1$ and $U_2$ minimise these indices, we have $$\frac{s(g)}{s(g^{-1})} \geq \frac{|U_1 : U_1 \cap g^{-1}U_1g |}{|U_1 : U_1 \cap g U_1 g^{-1}|} = \Delta (g) = \frac{|U_2 : U_2 \cap g^{-1} U_2 g|}{|U_2 : U_2 \cap gU_2g^{-1}|} \geq \frac{s(g)}{s(g^{-1})},$$ using Lemma~\ref{lemma-modular} to put the modular function in the middle of the inequality.
\end{proof}

Furthermore the proof of Proposition~\ref{prop-modular} shows (keeping notation) that $U_1$ and $U_2$ are minimising for \emph{both} $g$ and $g^{-1}$. So the next corollary follows directly.

\begin{corollary} \label{sameminimisingsubgroups}
For $g \in G$ and $U \in \mathcal{B}(G)$ we have $s(g) = |gUg^{-1}:gUg^{-1}\cap U|$ if and only if $s(g^{-1}) = |g^{-1}Ug:g^{-1}Ug\cap U|$.
\end{corollary}
 
The property that $g$ and $g^{-1}$ have the same minimising subgroups can be used to prove the following property of the scale function.

\begin{corollary} \label{normalise}
An element $g \in G$ normalises some $U \in \mathcal{B}(G)$ if and only if $s(g) = 1 = s(g^{-1})$.
\end{corollary}

\begin{proof}
The \enquote{only if} direction holds due to the definition of the scale function. For the \enquote{if} direction we use Corollary~\ref{sameminimisingsubgroups} to see that there exists $U \in \mathcal{B}(G)$ such that $|U : U \cap g^{-1}Ug| = 1= |U : U \cap gUg^{-1}|$. So $g^{-1}Ug$ and $gUg^{-1}$ both contain $U$ and hence are equal to U.
\end{proof}

\begin{example}
Let $G = \mathrm{Aut}(T_d)$ be the automorphism group of a regular tree $T_d$ of degree $d \geq 3$. Recall that an element $\varphi \in \mathrm{Aut}(T_d)$ either stabilises a point or an edge, in which case it is called elliptic, or there exists a unique axis along which $\varphi$ is a translation, and we say that $\varphi$ is hyperbolic. One easily see that $\varphi$ normalises a compact open subgroup of $\mathrm{Aut}(T_d)$ if and only if $\varphi$ is elliptic, in which case we have $s(\varphi) = s(\varphi^{-1}) = 1$. The scale function of a hyperbolic element will be computed in Section \ref{sec-ex}.
\end{example}

\subsection{Tidiness witnesses the scale}

Given $g \in G$, the properties shared by all the compact open subgroups at which the minimum is attained in the definition of the scale of $g$, are exactly the tidiness criteria.

\begin{theorem} \cite[Theorem~3.1]{Wil01} \label{thm-mini-tidy}
Let $U$ be a compact open subgroup of $G$ and $g \in G$. Then $U$ is minimising for $g$ if and only if $U$ is tidy for $g$.
\end{theorem}

\begin{proof}[On the proof]
We first assume that the forward implication is proved, and explain how to deduce the converse implication. The main argument is \cite[Lemma~5.20]{Wes}, which says that the quantity $|V : V \cap g^{-1}Vg|$ does not depend on the choice of $V$ as soon as $V$ is tidy for $\alpha$. Now choose a minimising subgroup $U$, and a tidy subgroup $V$. According to the forward implication of the statement the subgroup $U$ is also tidy, and by the previous lemma one has $|V : V \cap g^{-1}Vg| = |U : U \cap g^{-1}Ug| = s(g)$. Therefore $V$ is minimising as well, which proves the claim. 

The forward direction requires more work. This can be found in \cite[Theorem~3.1]{Wil01} or \cite[p.~39]{Wes}.
\end{proof}

\subsection{Subgroups being tidy above}

In this paragraph we give several characterisations of the notion of being tidy above, and explain how to construct subgroups tidy above for a given automorphism.

\begin{proposition}\thlabel{tachar}
For every $\alpha \in \aut$ and every $U \in \bg$, the following statements are equivalent:
\begin{enumerate}
\item $U=U_+U_-$, \label{pm}
\item $U$ is tidy above for $\alpha$,\label{ta}
\item $U=U_+\big(U\cap\alpha^{-1}(U)\big)$.\label{almostpm}
\end{enumerate}
\end{proposition}

\begin{proof}
As observed earlier, the map \[\chi\colon\alpha(U_+)/U_+\to\alpha(U)/\alpha(U)\cap U, \quad xU_+\mapsto x(\alpha(U)\cap U)\] is well-defined and injective. To see that (\ref{pm}) implies (\ref{ta}) we only have to show surjectivity of $\chi$. This follows immediately from \[\alpha (U)(\alpha (U)\cap U)=\alpha (U_+)\alpha (U_-)(\alpha (U)\cap U)\subseteq \alpha (U_+)(\alpha (U)\cap U)\]
and the disjointness of the cosets.\\
Now we show that (\ref{ta}) forces (\ref{almostpm}). Since $\chi$ is injective and by assumption the domain and the codomain have the same finite cardinality, $\chi$ has to be surjective. But this yields already \[\alpha (U)\subseteq \alpha (U)(\alpha (U)\cap U)=\alpha (U_+)(\alpha (U)\cap U),\] and by applying $\alpha^{-1}$ on both sides we get the claim.\\
In order to prove that (\ref{almostpm}) implies (\ref{pm}) we calculate by using the hypothesis \[\alpha^{-1}(U)\cap U=\alpha ^{-1}\big(U_+(\alpha^{-1}(U)\cap U)\big)\cap U_+(\alpha^{-1}(U)\cap U).\]
The last term is a set of the form \[L_1(M_1\cap N_1)\cap L_2(M_2\cap N_2),\quad L_i\subseteq N_i,\ M_2=N_1,\ L_1\subseteq L_2.\]
One can easily check that this set can be written as $L_1(M_1\cap M_2\cap N_2)$.
Using this we obtain the following expression for $U$:
\[U=U_+(\alpha^{-1}(U)\cap U)=U_+\alpha^{-1}(U_+)\underbrace{(\alpha^{-2}(U)\cap\alpha^{-1}(U)\cap U)}_{\alpha^{-1}(\alpha^{-1}(U)\cap U)\cap(\alpha^{-1}(U)\cap U)}\] and by induction $U=U_+\bigcap_{n=0}^k\alpha^{-n}(U)$ for every $k \geq 1$.

The conclusion then follows from a compactness argument, expressing the idea that $\bigcap_{n=0}^k\alpha^{-n}(U)$ tends to $U_-$ when $k$ goes to infinity (see \cite[Lemma~1]{Wil94} or \cite[Proposition~5.5]{Wes}).
\end{proof}

The proof of the following proposition yields an explicit procedure to exhibit subgroups that are tidy above for a given automorphism.

\begin{proposition}[Existence of subgroups being tidy above] \label{prop-exist-tdab}
For every $\alpha \in \Aut(G)$, there is a compact open subgroup that is tidy above for $\alpha$.
\end{proposition}

\begin{proof}
The idea is to start from an arbitrary compact open subgroup of $G$ and to consider $U_n =\bigcap_{i=0}^n\alpha^i(U)$ for $n \geq 0$. By construction there is an embedding \[\chi_n\colon\alpha (U_n)/\alpha (U_n)\cap U_n \hookrightarrow \alpha (U_{n-1})/\alpha (U_{n-1})\cap U_{n-1}.\] 
Hence these quotients form a chain with decreasing cardinalities. The compact open property of $U_0=U$ guarantees that $U_0/U_1$ is finite and therefore the chain has to become stationary. So there exists $N\in\mathbb N$ such that \[ |\alpha( U_j)/U_{j+1}|=|\alpha (U_N)/U_{N+1}| \] for every $j\geq N$. We will show that $V=U_N$ is tidy above. By the choice of $N$ the inclusions $\chi_{j,N}\colon\alpha( U_j)/\alpha (U_j)\cap U_j \hookrightarrow \alpha (U_{N})/\alpha (U_{N})\cap U_{N}$ are even bijections for all $j\geq N$. So we obtain by surjectivity of the map $\chi_{j,N}$ \[\alpha( V)\subseteq \alpha (U_N) U_{N+1}=\alpha(U_j)U_{N+1}=\alpha (U_j)(U_N\cap \alpha^{N+1}( U))=\alpha (U_j)(V\cap\alpha (V)).\]
A compactness argument similar to the one of the end of the proof of Proposition \ref{tachar}  yields $\alpha (V)=\alpha (V_+)(V\cap \alpha (V))$. Applying $\alpha^{-1}$ on both sides, we obtain that $V$ is tidy above for $\alpha$, which was the claim.
\end{proof}

An element $g \in G$ is called \textit{periodic} if $\overline{\langle g \rangle}$ is compact. We denote the set of all periodic elements of $G$ by $P(G)$. For example the reader can check as an exercise that $\varphi \in \mathrm{Aut}(T_d)$ is periodic if and only if $\varphi$ is elliptic.

The following theorem was proved by Willis using tidy subgroups and properties of the scale function such as continuity or Corollary \ref{normalise} (see \cite[Theorem~2]{Wil95}). Following \cite[Part~5]{Wes}, we present here a proof using only tidiness above.

\begin{theorem} \label{periodictheorem}
The set $P(G)$ is closed in $G$.
\end{theorem}

We will need the following easy lemma.

\begin{lemma} \label{powerprop}
If $U$ is tidy above for $g \in G$, then $(UgU)^n = Ug^nU$ and $(Ug^{-1}U)^n$ $=Ug^{-n}U$ for all $n \geq 1$.
\end{lemma}

\begin{proof}
It is sufficient to show the first equation, since $g$ and $g^{-1}$ have the same tidy above subgroups. We prove this equation by induction on $n$. For $n=1$ there is nothing to prove.
For $n+1$ we get $(UgU)^{n+1} = (UgU)^nUgU = (Ug^nU)gU$ by induction hypothesis. The last term equals to $Ug^nU_-U_+gU$ since $U$ is tidy above for $g$. Using the definition of $U_+$ and $U_-$ we see that $Ug^nU_- = Ug^n$ and $U_+gU = gU$. Altogether we conclude that $(UgU)^{n+1} = Ug^nU_-U_+gU = Ug^ngU = Ug^{n+1}U$, which completes the induction step.
\end{proof}

\begin{proof}[Proof of Theorem~\ref{periodictheorem}]
We give the proof under the additional assumption that $G$ is second countable, which allows us to use sequences to check that $P(G)$ is closed.

Let $(g_i)_{i\geq 0}$ be a sequence in $P(G)$ with limit $g \in G$. We need to prove that $g$ is contained in $P(G)$. Let $U \in \mathcal{B}(G)$ be tidy above for $g$, and fix $i\geq 0$ such that $g_i \in UgU$ (which is an open neighbourhood of $g$). Then there are $u, v \in U$ with $g_i = ugv$. 

Now we want to prove the existence of an integer $n \geq 0$ such that $g_i^n \in U$. For this consider $(g_i^j)_{j \geq 0}$ which is a sequence in the compact set $\overline{\langle g_i \rangle}$  and thus has a convergent subsequence. Let $(g_i^{j_k})_{k \geq 0}$ be this convergent subsequence and $h$ its limit point. Since $h \in \overline{\langle g_i \rangle}$, there is an integer $m \in \mathbb{N}$ such that $ h \in g_i^m U$, which is an open neighbourhood of $g_i^m$. So we can build a new sequence $(g_i^{j_k-m})_{k\geq 0}$ converging to $g_i^{-m}h \in U$ which gives us the existence of the desired $n$ with $g_i^n = (ugv)^n\in U$.

Due to Lemma~\ref{powerprop}, there are $u', v' \in U$ such that $u'g^nv'=(ugv)^n \in U$, so already $g^n$ has to be contained in $U$. Therefore $\overline{\langle g^n \rangle}$ is compact as a closed subset of the compact set $U$, and $\overline{\langle g\rangle}$ is compact, since it has $\overline{\langle g^n \rangle}$ as compact subgroup of finite index. Hence $g$ is periodic.
\end{proof}

In \cite{Wil95}, Willis mention the following application of this result. In 1981 Hofmann asked if, for a general group $G$ such that $P(G)$ is dense in $G$, already $P(G) = G$ holds. In general the answer is no. A counterexample is the group $E$ of motions of the Euclidean plane, for which one can check that $P(E)$ consists of all rotations and reflections, and is a proper dense subgroup of $E$. From Theorem~\ref{periodictheorem} we can deduce a positive answer to Hofmann's question for totally disconnected locally compact groups.

\subsection{Construction of tidy subgroups}

In this paragraph we explain how to construct tidy subgroups. Roughly the idea of the procedure is, given a subgroup that is tidy above for $g \in G$, to add a $g$-invariant compact subgroup. 

For $g \in G$ and $U \in \bg$, we let \[ \preline = \left\{x \in G \, : \, x \in g^n U g^{-n} \, \, \text{for all but finitely many $n \in \mathbb{Z}$} \right\}. \]

The following proof appears in \cite[Lemma~5.14]{Wes}.

\begin{lemma} \label{lemma-L-rel-cpt}
The subset $\preline$ is a subgroup normalised by $g$ that is relatively compact. 
\end{lemma}

\begin{proof}
The verification of the fact that $\preline$ is a subgroup normalised by $g$ is easy, and we leave it to the reader. We prove that $\preline$ is contained in a compact subset of $G$.

Since $g U_+ g^{-1}$ is compact and $U$ is open, the intersection $g U_+ g^{-1} \cap U$ has finite index in $g U_+ g^{-1}$. But this intersection is equal to $U_+$ by definition, and it follows that $U_+ \cap \preline$ has finite index in $g U_+ g^{-1} \cap \preline$. We let $X$ be a finite set of left coset representatives, so that \[ g U_+ g^{-1} \cap \preline = X \left( U_+ \cap \preline \right). \]
We claim that for every $n \geq 1$, one has \begin{equation} \label{eq-L-rl-comp} g^n U_+ g^{-n} \cap \preline = \left( \prod_{i=n-1}^0 g^i X g^{-i} \right) U_+ \cap \preline. \end{equation}
We argue by induction. The case $n=1$ is nothing but the definition of the set $X$. Assume that (\ref{eq-L-rl-comp}) holds for some $n \geq 1$. Since $\preline$ is normalised by $g$, we have \[ g^{n+1} U_+ g^{-(n+1)} \cap \preline = g \left(  g^n U_+ g^{-n} \cap \preline \right) g^{-1}. \] Using the induction hypothesis we obtain \begin{align*} g^{n+1} U_+ g^{-(n+1)} \cap \preline & =g \left( \prod_{i=n-1}^0 g^i X g^{-i} \right) (U_+ \cap \preline) g^{-1} \\ & = \left( \prod_{i=n-1}^0 g^{i+1} X g^{-(i+1)} \right) g (U_+ \cap \preline) g^{-1} \\ & = \left( \prod_{i=n}^1 g^{i} X g^{-i} \right) X (U_+ \cap \preline), \end{align*} where the last equality follows from the definition of $X$. So we obtain that \[ g^{n+1} U_+ g^{-(n+1)} \cap \preline = \left( \prod_{i=n}^0 g^{i} X g^{-i} \right) U_+ \cap \preline, \] and the proof of the induction step is complete.

Since $X$ is finite, there exists some integer $n_0 \geq 1$ such that $X$ lies inside $g^{-n_0} U_- g^{n_0}$. It follows that for every $n \geq n_0$, we have \[ g^n X g^{-n} \subset g^{n-n_0} U_- g^{-(n-n_0)} \subset U_- \subset  U. \] Using (\ref{eq-L-rl-comp}) we obtain that for every $n \geq n_0$, \[ g^n U_+ g^{-n} \cap \preline \subset U \left( \prod_{i=n_0-1}^0 g^i X g^{-i} \right) U_+ \cap \preline \subset U \left( \prod_{i=n_0-1}^0 g^i X g^{-i} \right) U =: K. \] We obtain that $\preline$, which is the increasing union for $n \geq n_0$ of $g^n U_+ g^{-n} \cap \preline$, is contained in $K$. The latter being compact since $X$ is finite and $U$ is compact, this finishes the proof.
\end{proof}

For $g \in G$ and $U \in \bg$, we denote by $L_U$ the closure of $\preline$ in $G$. The reason why the definition of $\preline$ is relevant in this context comes from the following result.

\begin{proposition} \thlabel{chartidy}
Let $g \in G$ and $U \in \bg$ being tidy above for $g$. Then $U$ is tidy for $g$ if and only if $L_U \leq U$.
\end{proposition}

\begin{proof}
See for instance \cite[Proposition~5.15]{Wes}.
\end{proof}

The following gives a recipe to construct tidy subgroups.

\begin{proposition}\thlabel{existtidy}
For every $g \in G$, there exists $U \in \bg$ that is tidy for~$g$.
\end{proposition}

\begin{proof}
The strategy consists essentially in two steps.

\medskip

\textbf{Step 1}. Given a compact open subgroup $U$ of $G$, there exists $n \geq 0$ such that $\bigcap_{i=0}^n g^i U g^{-i}$ is tidy above for $g$. The existence of such an integer has been shown in the proof of Proposition~\ref{prop-exist-tdab}.

\medskip

\textbf{Step 2}. Assume we have a compact open subgroup $U$ tidy above for $g$. Let $U' = \bigcap_{x \in L_U} xUx^{-1}$. Since $L_U$ is compact according to Lemma~\ref{lemma-L-rel-cpt}, it follows that $U' = \bigcap_{i=1}^k x_i U x_i ^{-1}$ for some elements $x_1, \ldots, x_k \in L_U$, and therefore $U'$ is open in $G$. Since by definition $U'$ is normalised by $L_U$, it follows that $U'L_U$ is a compact open subgroup of $G$. Now going back to Step 1, we can find some integer $n \geq 0$ such that $U'' = \bigcap_{i=0}^n g^i U'L_U g^{-i}$ is tidy above for $g$. Since $L_U$ is normalised by $g$, clearly $L_U \subset U''$. We refer the reader to \cite[Proposition~5.17]{Wes} or \cite[Lemmas 3.3-3.8]{Wil01} to see that this property implies that $U''$ is also tidy below for $g$.
\end{proof}

\section{Examples} \label{sec-ex}

The goal of this section is to illustrate with several examples the notions and results from the previous section.

\subsection{Infinite direct product of a finite group}

Let $F$ be a non-trivial finite group and $G=F^{\mathbb Z}$ endowed with the product topology. We define $\alpha\in\Aut(G)$ by $\alpha((f_n))=(f_{n+1})$. Since $G$ is compact, it is itself a tidy subgroup for $\alpha$. For every subset $I\subseteq \mathbb Z$, define \[ E_I=\{(f_n) \in G \, : \, f_n = e \, \, \text{for every} \, \, n \in I\}.\] Let us consider the compact open subgroup $U= E_{[-5,3]}$. By definition we have $U_-=E_{[-5,\infty[}$ and $U_+=E_{]-\infty,3]}$. We observe that \[U\subseteq E_{]-\infty,3]} E_{[-5,\infty[} =U_+U_-\subseteq U\] and therefore $U$ is tidy above for $\alpha$. See Figure~\ref{fig:singlegap}.

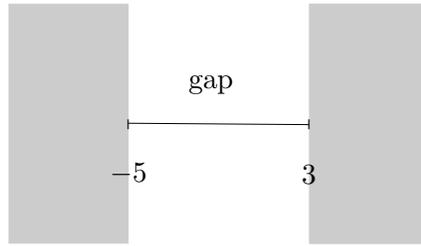
\begin{figure}[h!t!]

\psscalebox{1.0 1.0} 
{
\begin{pspicture}(0,-1.64)(5.6419997,1.64)
\definecolor{colour1}{rgb}{0.8,0.8,0.8}
\psframe[linecolor=white, linewidth=0.002, fillstyle=gradient, gradlines=2000, gradbegin=colour1, gradend=colour1, dimen=inner](1.6104972,1.6021274)(0.0019995116,-1.6)
\psframe[linecolor=white, linewidth=0.04, fillstyle=gradient, gradlines=2000, gradbegin=colour1, gradend=colour1, dimen=inner](5.6019993,1.6)(4.0019994,-1.6)
\psline[linecolor=black, linewidth=0.01](1.6020288,0.0076600458)(4.0019584,-0.010724064)
\rput[bl](2.4019995,0.4){gap}
\psdots[linecolor=black, dotstyle=|, dotsize=0.1](1.6019995,0.0)
\psdots[linecolor=black, dotstyle=|, dotsize=0.1](4.0019994,0.0)
\rput[b](1.6019995,-0.8){$-5$}
\rput[b](4.0019994,-0.8){$3$}
\end{pspicture}
}

\caption{This figure illustrates the subgroup $U=E_{[-5,3]}$ of $F^{\mathbb{Z}}$. The fact that $U$ is tidy above corresponds to the existence of only one \enquote{gap}.
}
\label{fig:singlegap}
\end{figure}

More generally, if $I$ is a finite non-empty subset of $\mathbb{Z}$ whose minimum and maximum are denoted respectively by $a$ and $b$ and $U = E_I$, then we observe that $U_- = E_{[a, \infty[}$ and $U_+ = E_{]-\infty, b]}$, hence $U_+ U_- = E_{[a, b]}$. Therefore $U$ is tidy above for $\alpha$ if and only if $I = [a, b]$. Following the strategy used in the proof of Proposition~\ref{prop-exist-tdab}, we see that the intersection $\bigcap_{i=0}^n \alpha^i (E_I)$ will indeed be equal to some $E_{[a', b]}$ for $n$ large enough (see Figure~\ref{fig:multigaps}). 

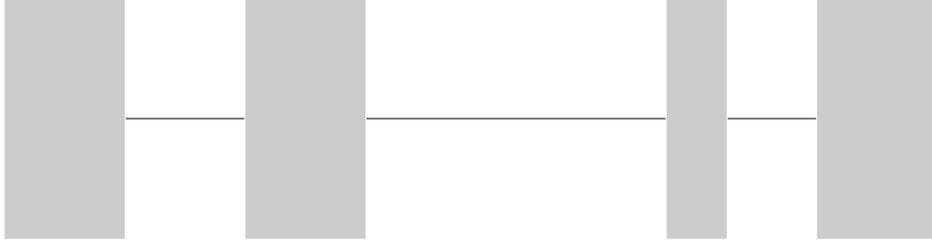
\begin{figure}[h!t!]
\psscalebox{1.0 1.0} 
{
\begin{pspicture}(0,-1.61)(12.42,1.61)
\definecolor{colour0}{rgb}{0.8,0.8,0.8}
\psline[linecolor=black, linewidth=0.01](1.61,0.0)(2.6766667,0.0)(3.21,0.0)
\psline[linecolor=black, linewidth=0.01](4.81,0.0)(8.446363,0.0)(8.81,0.0)
\psline[linecolor=black, linewidth=0.01](9.61,0.0)(10.81,0.0)(10.81,0.0)
\psframe[linecolor=white, linewidth=0.01, fillstyle=gradient, gradlines=2000, gradbegin=colour0, gradend=colour0, dimen=inner](1.61,1.6)(0.01,-1.6)
\psframe[linecolor=white, linewidth=0.01, fillstyle=gradient, gradlines=2000, gradbegin=colour0, gradend=colour0, dimen=inner](4.81,1.6)(3.21,-1.6)
\psframe[linecolor=white, linewidth=0.01, fillstyle=gradient, gradlines=2000, gradbegin=colour0, gradend=colour0, dimen=inner](9.61,1.6)(8.81,-1.6)
\psframe[linecolor=white, linewidth=0.01, fillstyle=gradient, gradlines=2000, gradbegin=colour0, gradend=colour0, dimen=inner](12.41,1.6)(10.81,-1.6)
\end{pspicture}
}
\caption{A compact open subgroup of $F^{\mathbb{Z}}$ which is not tidy above for $\alpha$. The successive intersections $\bigcap_{i=0}^n \alpha^i (E_I)$ will gradually exhibit fewer and fewer \enquote{gaps}.}
\label{fig:multigaps}
\end{figure}

Furthermore we observe that $U_{++}$ (resp.\ $U_{--}$) is the subgroup of sequences that are trivial below (resp.\ above) some index, which is a proper dense subgroup, hence no $E_I$ is tidy below for $\alpha$. We can also see that $\mathcal{L}_U$ is the dense subgroup of sequences with finite support, which is another reason why $E_I$ cannot be a tidy subgroup by Proposition~\ref{chartidy}. Moreover, we observe that the tidy subgroup given by the recipe of Proposition~\ref{existtidy} is the whole group $G$.

\subsection{A $p$-adic Lie group}

We let $p$ be a prime number, $G = \mathrm{GL}_2(\mathbb{Q}_p)$ be the group of invertible $2 \times 2$ matrices over the field of $p$-adic numbers, and $g \in G$ be the diagonal matrix $g = \mathrm{diag}(p,1)$. We follow the recipe to construct a compact open subgroup of $G$ tidy for $g$. Let us start with the compact open subgroup $U = \mathrm{SL}_2(\mathbb{Z}_p)$. Since conjugation by $g^n$, $n \in \mathbb{Z}$, multiplies the upper right entry of a matrix by $p^n$ and the lower left entry by $p^{-n}$, it follows that \[ U_+ = \begin{pmatrix} * & 0 \\ * & *\end{pmatrix} \cap \mathrm{SL}_2(\mathbb{Z}_p) \, \, \text{and} \, \, U_- = \begin{pmatrix} * & * \\ 0 & *\end{pmatrix} \cap \mathrm{SL}_2(\mathbb{Z}_p). \]
Now we claim that $U_+ U_- \subsetneq U$. Indeed, argue by contradiction and assume that $U_+ U_- = U$. Using the natural morphism $\mathrm{SL}_2(\mathbb{Z}_p) \twoheadrightarrow \mathrm{SL}_2(\mathbb{F}_p)$, we obtain that a similar equality holds in the finite group $\mathrm{SL}_2(\mathbb{F}_p)$, where $\mathbb{F}_p$ is the field with $p$ elements. But this is impossible for counting reasons, since the set of products $xy$ in $\mathrm{SL}_2(\mathbb{F}_p)$ with $x$ lower triangular and $y$ upper triangular has cardinality $(p-1)p^2$, whereas the group $\mathrm{SL}_2(\mathbb{F}_p)$ has cardinality $(p-1)p(p+1)$.

So it follows that $U = \mathrm{SL}_2(\mathbb{Z}_p)$ is not tidy above for $g$. Following the construction on the proof of Proposition \ref{prop-exist-tdab}, we let \[ U' = U \cap g U g^{-1} = \begin{pmatrix} * & p \mathbb{Z}_p \\ * & * \end{pmatrix} \cap \mathrm{SL}_2(\mathbb{Z}_p). \] Similarly $U_+'$ and $U_-'$ are the sets of matrices of determinant one of the form \[ U_+' = \begin{pmatrix} \mathbb{Z}_p^{\times} & 0 \\ \mathbb{Z}_p & \mathbb{Z}_p^{\times} \end{pmatrix} \, \, \text{and} \, \, U_-' = \begin{pmatrix} \mathbb{Z}_p^{\times}  & p \mathbb{Z}_p \\ 0 & \mathbb{Z}_p^{\times} \end{pmatrix}.\]
Now observe that the upper left entry of any element of $U'$ must be a unit in $\mathbb{Z}_p$, and a trivial computation shows that \[ \begin{pmatrix} a & p b \\ c & d \end{pmatrix} = \begin{pmatrix} a & 0 \\ c & a^{-1} \end{pmatrix} \begin{pmatrix}1  & p b a^{-1} \\ 0 & 1 \end{pmatrix}, \] so one has $U' = U_+' U_-'$ and $U'$ is tidy above for $g$. Moreover \[ U_{++}' = \begin{pmatrix} \mathbb{Z}_p^{\times} & 0 \\ \mathbb{Q}_p & \mathbb{Z}_p^{\times} \end{pmatrix} \, \, \text{and} \, \, U_{--}' = \begin{pmatrix} \mathbb{Z}_p^{\times}  & \mathbb{Q}_p \\ 0 & \mathbb{Z}_p^{\times} \end{pmatrix}\] are closed in $G$, so $U'$ is also tidy below for $g$. Therefore $U'$ is minimising for $g$, and it follows that the scale of $g$ is the index of $U_+'$ in $gU_+'g^{-1}$, which is equal to $p$.

\subsection{The automorphism group of a regular tree} 
We let $G = \mathrm{Aut}(T_d)$ be the automorphism group of the regular tree $T_d$ of degree $d \geq 3$, and we compute the scale function and describe tidy subgroups of any element of $G$ (see also \cite[Section 3]{Wil94}).

If $\varphi \in G$ is elliptic then $\varphi$ normalises either an edge-stabiliser or a vertex-stabiliser. Such a compact open subgroup is necessarily tidy for $\varphi$, and $\varphi$ has scale one.

Now assume that $\varphi$ is a hyperbolic element whose axis is defined by the sequence of vertices $(v_n)$, $n \in \mathbb{Z}$, indexed so that $\varphi$ translates in the positive direction. We denote by $\xi^{-}, \xi^{+} \in \partial T_d$ the repelling and attracting endpoints of $\varphi$. Let $v$ be a vertex of $T_d$, and let $U \in \bg$ be the stabiliser of $v$. Then $U_+$ (resp.\ $U_-$) is the subgroup of elements fixing pointwise the set of vertices $\varphi^n(v)$ for $n \geq 0$ (resp.\ $n \leq 0$). In particular any element of $U_+U_-$ fixes pointwise the unique geodesic between $v$ and the axis of $\varphi$. It follows that if $v$ does not lie on the axis of $\varphi$ then $U$ cannot be tidy above for $\varphi$. When $v$ lies on the axis of $\varphi$, a similar argument shows that an element of $U_+U_-$ cannot send the edge emanating from $v$ and pointing toward $\xi^{-}$ to the one pointing toward $\xi^{+}$, so again $U_+U_- \subsetneq U$. 

Now let $U$ be the stabiliser of $v_0$ for example, and let $U' = U \cap \varphi U \varphi^{-1}$ be the stabiliser of the geodesic between $v_0$ and $v_{\ell}$, where $\ell \geq 1$ is the translation length of $\varphi$. Note that any element of $U'$ fixes the two half-trees of $T_d$ obtained by cutting the edge $(v_0,v_1)$. Again $U_+'$ (resp.\ $U_-'$) is the pointwise stabiliser of the geodesic ray $(\varphi^n v_0)_{n \geq 0}$ (resp.\ $(\varphi^n v_{\ell})_{n \leq 0}$), and an easy verification shows that $U' = U_+' U_-'$. Moreover $U_{++}'$ (resp.\ $U_{--}'$) is the subgroup of the stabiliser in $G$ of $\xi^{+}$ (resp.\ $\xi^{-}$) consisting of elliptic isometries, and therefore is closed. So $U'$ is tidy for $g$.

A direct computation shows that the stabiliser of $v_{2 \ell}$ has index $(d-1)^{\ell}$ in $U'$, so it follows that the scale of $\varphi$ is equal to $(d-1)^{\ell}$.

\section{Ergodic automorphisms of locally compact groups}

We conclude with an application of tidy subgroups to ergodic theory, due to Previts and Wu \cite{Prev-Wu}. In 1955 Halmos asked the following question (see \cite[p.~29]{Halmos}): \enquote{\emph{Can an automorphism of a locally compact but non-compact group be an ergodic measure-preserving transformation?}}. Recall that a transformation $T$ is called \emph{ergodic} relatively to a measure $\mu$ if for any measurable set $E$ such that $\mu (T^{-1} E \triangle E) = 0$, one has $\mu (E) = 0$ or $\mu (E^c) = 0$. Here, ergodicity has to be understood with respect to a left-invariant Haar measure on the group.

The answer to Halmos's question is negative, as it was progressively proved by Juzvinski\u{i}, Rajagopalan, Kaufman, Wu, and Aoki in several works ranging from 1965 to 1985. See the introduction of \cite{Prev-Wu} for more details on this topic. We give here only a part of this solution, namely the result of Aoki for the totally disconnected case, following the proof given in \cite{Prev-Wu}.

\begin{theorem}[{\cite[Theorem~1]{Aoki}}] \label{MGALB:thm_halmos_aoki}
 A totally disconnected locally compact group admitting an ergodic automorphism must be compact.
\end{theorem}

Conversely, a compact group may have ergodic automorphisms: consider the Bernoulli shift $\sigma$ on the profinite group ${(\mathbb{Z} / 2 \mathbb{Z})}^{\mathbb{Z}}$, defined by $\sigma((g_n)) = (g_{n - 1})$.

It is clear that a non-trivial discrete group cannot have an ergodic automorphism, so we may assume the group under consideration to be non-discrete.

Before going to the proof of Theorem~\ref{MGALB:thm_halmos_aoki}, let us mention that, by a result of Rajagopalan (\cite[Theorem~1]{Raja}), a continuous ergodic automorphism of a locally compact group is automatically bicontinuous.

\subsection{Reduction of the problem} We will first gradually reduce the problem to a more tractable case. More precisely, we will show that we may assume the group to be second countable and the automorphism to have a dense orbit.

\begin{lemma}
 If a locally compact group $G$ admits an ergodic automorphism $f$, then $G$ is $\sigma$-compact.
\end{lemma}

\begin{proof}
 Indeed, let $V$ be a symmetric compact neighbourhood of the identity. The subgroup $H = \langle V \rangle = \bigcup_n V^n$ is obviously open and $\sigma$-compact. Hence so is the subgroup $H'$ generated by the $f$-translates of $H$. But $H'$ is clearly $f$-invariant, hence $H' = G$ by ergodicity.
\end{proof}

\begin{lemma}
 To prove Theorem~\ref{MGALB:thm_halmos_aoki}, we may assume that the group is separable and metrisable (in particular, second countable).
\end{lemma}

\begin{proof}
 By the previous lemma, the group $G$ is $\sigma$-compact. Therefore, by the Kakutani-Kodaira theorem (see \cite[Theorem~8.7]{HR}), there exists a normal subgroup $K$ such that the quotient $G / K$ is separable and metrisable. Let $K_n = \bigcap_{-n}^n f^j (K)$. The quotient $G / K_n$ embeds into the product
\begin{equation*}
 \prod_{j = -n}^n G / f^j (K),
\end{equation*}
hence is also separable and metrisable. Therefore, so is $G' = G / K_\infty$ (where $K_\infty = \bigcap_{-\infty}^\infty f^j (K)$), which is the inverse limit of the $G / K_n$. Obviously, the automorphism $f$ descends to $G'$ and is still ergodic; moreover, $G$ is compact if and only if $G'$ is so (since $K_\infty$ is compact).
\end{proof}

\begin{lemma}
 If $f$ is an ergodic automorphism of a locally compact second countable group, then $f$ has a dense orbit.
\end{lemma}

\begin{proof}
Let $\{O_n\}$ be a countable basis of open sets. Suppose by contradiction that no point of $G$ has a dense orbit. Then for any $x \in G$, there is an integer $n_x$ such that $x$ does not belong to the orbit of $O_{n_x}$. The latter is an $f$-invariant set containing an open subset, hence its complement has measure zero, by ergodicity. Therefore $G$ could be written as the countable union of the $G \setminus \bigcup_j f^j (O_n)$, hence would have measure zero, which is absurd.
\end{proof}

\subsection{Tidy subgroups come into play} 
We will need the following result about tidy subgroups.

\begin{proposition} \label{MGALB:prop_willis_clopen}
Let $G$ be a totally disconnected locally compact group. If $U$ is a tidy subgroup for $\alpha \in \aut$, then the set $U^\star = \bigcup_{i \geqslant 0} \alpha^i (U)$ is (open and) closed.
\end{proposition}

\begin{proof}
 See \cite[Proposition~1]{Wil94}.
\end{proof}

We finally need the following easy lemma.

\begin{lemma} \label{MGALB:lemma_ergodic_clopen}
 Let $X$ be a locally compact non-discrete space and $f$ an automorphism with a dense orbit. If $A \subseteq X$ is a non-empty open and closed subset such that $f(A) \subseteq A$, then $A = X$.
\end{lemma}

\begin{proof}
 Let $x$ be a point in $A$ with dense orbit, which exists because $A$ is open. By density of the orbit and non-discreteness of the space, there is a strictly monotone sequence of integers $n_i$ such that $x$ is the limit of $f^{n_i} (x)$. We may assume all the $n_i$ to be of the same sign, say positive (the negative case being proved similarly). Let $k \in \mathbb{N}$. As $f(A) \subseteq A$, $f^k (x) \in A$. Moreover, $f^{-k} (x)$ is the limit of the points $f^{n_i - k} (x)$, which are in $A$ for $i$ big enough. Hence $f^{-k} (x)$ is also in $A$, as the latter is closed. Therefore $A$ is a closed set containing the orbit of $x$, which is dense, hence $A = X$.
\end{proof}

\begin{proof}[Proof of Theorem~\ref{MGALB:thm_halmos_aoki}]
Let $U$ be a tidy subgroup for the automorphism $f$. By Proposition~\ref{MGALB:prop_willis_clopen}, $U^\star = \bigcup_{i \geqslant 0} f^i (U)$ is closed and open. By Lemma~\ref{MGALB:lemma_ergodic_clopen}, $U^\star = G$. Therefore, by compactness of $U$, $f^{-1} (U)$ must lie inside $U_m = \bigcup_{i \geqslant 0}^m f^i (U)$ for some $m$. Thus we have
\begin{equation*}
 U_m \subseteq U_{m + 1} = U \cup f(U_m) \subseteq f(U_m)
\end{equation*}
 and Lemma~\ref{MGALB:lemma_ergodic_clopen} again implies that $G = U_m$, hence $G$ is compact.
\end{proof}

\bibliographystyle{amsalpha}
\bibliography{tidy}

\end{document}